\documentclass[10pt,a4paper]{article}

\usepackage{amsmath,amsfonts,amsthm,amssymb,amscd, appendix, verbatim}
\usepackage{hyperref}
\hypersetup{colorlinks=true, linkcolor=blue, citecolor=blue, linktoc=page, pdftitle={Weak KAM}, pdfauthor={Figalli, Gomes, Marcon}}
\usepackage[latin1]{inputenc}
\usepackage{setspace, color}
\usepackage{array}

\newtheorem{theorem}{Theorem}[section]
\newtheorem{corollary}[theorem]{Corollary}

\newtheorem{lemma}[theorem]{Lemma}

\theoremstyle{definition}

\newtheorem{remark}[theorem]{Remark}

\newcommand{\R}{\mathbb{R}}

\newcommand{\Lap}{\Delta}

\newcommand{\eps}{\varepsilon}

\newcommand{\fr}{\partial}
\newcommand{\grad}{\nabla}

\newcommand{\Div}{\operatorname{div}}

\newcommand{\norm}[1]{\left\Vert#1\right\Vert}

\usepackage[margin=1in]{geometry}

\title{On the Blow-up Criterion of Magnetohydrodynamics Equations in Homogeneous Sobolev Spaces}
\author{Diego Marcon,  \hspace*{.2cm} Wilberclay G. Melo, \hspace*{.2cm} Lineia Schutz, \hspace*{.2cm} and Juliana S. Ziebell}
\date{}

\begin{document}

\maketitle

\def\signdm{\bigskip\begin{center} {\sc Diego Marcon\par\vspace{3mm}
Instituto de Matemática\\
Universidade Federal do Rio Grande do Sul (UFRGS)\\
Porto Alegre 91509-900, Brazil\\
email:} {\tt diego.marcon@ufrgs.br}
\end{center}}

\def\signwm{\bigskip\begin{center} {\sc Wilberclay G. Melo\par\vspace{3mm}
Departamento de Matemática\\
Universidade Federal de Sergipe (UFS) \\
São Cristóvão 49100-000,  Brazil. \\
email:} {\tt wilberclay@gmail.com}
\end{center}}

\def\signls{\bigskip\begin{center} {\sc Lineia Schutz\par\vspace{3mm}
Instituto de Matemática\\
Universidade Federal do Rio Grande do Sul (UFRGS)\\
Porto Alegre 91509-900, Brazil\\
email:} {\tt lineia.schutz@ufrgs.br}
\end{center}}

\def\signjz{\bigskip\begin{center} {\sc Juliana S. Ziebell\par\vspace{3mm}
Instituto de Matemática, Estatística e Física\\
Universidade Federal do Rio Grande (FURG)\\
Rio Grande 96201-900, Brazil\\
email:} {\tt julianaziebell@furg.br}
\end{center}}

\begin{abstract}
In this paper, we obtain new lower bounds on the Homogeneous Sobolev--norms of the maximal solution of the Magnetohydrodynamics Equations. This gives us some insight on the blow-up behavior of the solution. We utilize standard techniques from the Navier--Stokes Equations.
\end{abstract}



\section{Introduction}

Magnetohydrodynamics (MHD) describes the motion of electrically conducting fluids, such as liquid metals, salt water, or plasmas, in the presence of magnetic fields. The MHD Equations form a system of equations that combine fluid dynamics with Maxwell's theory of electromagnetism. For an introduction on the subject, see \cite{schnack}. In this paper, we consider the three--dimensional MHD Equations for incompressible flows:
\begin{equation}\label{MHD2}
\left\{
  \begin{array}{ll}
    \displaystyle \frac{\fr u}{\fr t}  + (u \cdot \nabla ) u = \nu \Delta u + \sigma (b \cdot \nabla) b - \nabla \Big( p + \frac{1}{2}\sigma |b|^2 \Big), & \hbox{ in } \R^+ \times \R^3; \\
    & \\
    \displaystyle \frac{\fr b}{\fr t}  + (u \cdot \nabla ) b = \nu_{m} \Delta b + (b \cdot \nabla ) u, & \hbox{ in } \R^+ \times \R^3; \\
    & \\
    \displaystyle \Div u = 0, \quad \Div b = 0, & \hbox{ in } \R^+ \times \R^3; \\
    & \\
    \displaystyle u(0) = u^0 , \quad b(0) = b^0  \ \mbox{ with } \Div u^0 = 0, \qquad \Div b^0 = 0, & \hbox{ in } \R^3,
  \end{array}
\right.
\end{equation} where $u(x,t)= \big(u_{1}(x,t),u_{2}(x,t), u_{3}(x,t) \big) \in \R^3$  denotes the velocity field of the fluid and $b(x,t)= \big( b_{1}(x,t), b_{2}(x,t), b_{3}(x,t) \big) \in \R^3 $ the magnetic field. The remaining terms are the hydrostatic pressure $p(x,t) \in \R$, the initial data $u_{0}$, $b_{0}$ in $L^{2}(\R^3; \R^3)$, and the positive constants $\nu$, $\nu_m$, and $\sigma$. These constants denote, respectively, the kinematic constant $\nu = 1/R_e$, the magnetic diffusivity constant $\nu_{m}= 1/R_m$, and $\sigma= M^2/R_eR_m,$ where $R_e$ is the Reynolds number, $R_m$ is the magnetic Reynolds number, and $M$ is the Hartman number.

The existence and the uniqueness of solutions of MHD systems have been extensively studied \cite{Bealekatomajda, caflisch, DL, hexin2005, schonbek, sermange, Yuan2008, Yuan2010, zhou2005}. Duvaut and Lions \cite{DL} introduced a class of weak solutions with finite energy. In the two--dimensional case, the existence of a unique classical solution has been proved  by Sermange and Teman \cite{sermange}. Whether smooth solutions exist in the three--dimensional case is a major open problem.  In general, one can only ensure the existence of a maximal time $T^* >0$ for which the MHD system \eqref{MHD2} has a classical solution $u(x,t), b(x,t)$ defined for $(x,t) \in [0, T^*) \times \R^3.$ This is the analog of the corresponding Incompressible Navier--Stokes problem. In fact, in the absence of a magnetic field, the MHD Equations reduce to the Navier--Stokes Equations. For regularity results and blow--up estimates of the Navier--Stokes Equations, see \cite{B2010, Kato1984, KreissHagstromLorenzZingano2003, Leray1934, Masuda1984, Robinson2012, Wiegner1987}.

For Homogeneous Sobolev Spaces $H^{s_{0}}(\R^3)$, Yuan \cite{Yuan2008} obtained the existence of strong solutions of \eqref{MHD2} in $[0, T^{\ast}) \times \R^3$. More precisely, for $s_{0} > 3/2$ and for divergence--free $u^{0}, b^{0} \in H^{s_{0}}(\R^3)$, there exists a positive time $T^\ast > 0$, depending on the initial conditions $u^0, b^0$, such that \eqref{MHD2} admits a unique classical solution $(u, b)$, defined on $[0, T^*) \times \R^3$. Motivated by He and Xin \cite{hexin2005}, Yuan also provides blow-up criteria that is independent of the magnetic field $b$, see \cite[Theorem 1.1]{Yuan2008}.

In this paper, we suppose $T^{\ast} < \infty$ and  we establish properties of the maximal solution of \eqref{MHD2}. Inspired by Benameur \cite{B2010}, we obtain blow--up estimates for strong solutions of the Incompressible MHD Equations in the Homogeneous Sobolev Space $\dot{H}^{s}(\R^3)$, for $s>1/2$.

Set $\lambda = \min \{\nu, \nu_{m}\}$ and $\kappa=\max\{1,\sigma\}^{-1}$. Our main theorem is the following:

\begin{theorem}\label{mainresult}
Fix $s_{0}> 3/2$ and let $u^0$, $b^0 \in H^{s_{0}}(\R^3)$ be such that $\Div u^{0} = \Div b^{0} = 0$.
Consider $u$, $b \in C\big([0,T^*), H^s(\R^3)\big)$ the strong solution of \eqref{MHD2}, defined in its maximal interval of definition $[0, T^*)$.  If $T^* < + \infty$, then
\begin{itemize}
\item[$(i)$] For every $\delta \in(0,1)$ and for every $s\ge1/2+\delta$, we have
\begin{equation}\label{i}
 \norm{\big(u,b\big)(t)}_{\dot{H}^s(\R^3)} \norm{\big(u,b\big)(t)}_{L^2( \R^3)}^{p(s,\delta)} \ge \frac{c(s, \delta,\sigma)\lambda^{q(s,\delta)}}{(T^*-t)^{r(s,\delta)}},
\end{equation} where $$p(s,\delta)  :=  \frac{2s}{1+2\delta} - 1, \quad  q(s,\delta)  :=  \frac{(2-\delta)s}{1+2\delta}, \quad \hbox{ and } r(s,\delta)  :=  \frac{s\delta}{1 + 2\delta}.$$
\item[$(ii)$] For all $t\in[0,T^*)$,
\begin{equation}\label{fourier_thm}
\norm{\mathcal{F}\big(u(t),b(t)\big)}_{L^1(\R^3)}\geq \frac{\kappa (2\pi)^{3}\lambda^{1/2}}{ 3\sqrt{6}}{(T^{*}-t)^{-1/2}},
\end{equation}

\item[$(iii)$] For every $s>\frac{3}{2}$,
\begin{equation}
\norm{(u,b)(t)}_{L^2( \R^3)}^{\frac{2s}{3}-1}\norm{(u,b)(t)}_{\dot{H}^s( \R^3)}\geq\frac{c(s,\sigma)\lambda^{s/3}}{(T^*-t)^{s/3}}.
\end{equation}

\end{itemize}
\end{theorem}

\begin{remark}
We observe that Theorem \ref{mainresult} extends and improves the results of \cite{B2010}. If $b=0$, the MHD Equations become the Navier--Stokes Equations. In this case, the choice $\delta= 1/2$ provides us with Benameur's estimate, see \cite[Theorem 1.3, (1.2)]{B2010}. Moreover, our approach allows us to consider any $s > 3/2$ whereas in \cite{B2010}, the author assumes $s>5/2$.
\end{remark}

\begin{remark}
We note that, in \eqref{fourier_thm} above, $\norm{\mathcal{F}(u,b)(t)}_{L^1}$ is finite, for all $t \in [0, T^\ast)$. Indeed, we only need to observe that, for any $v\in \dot{H}^{s}$,
\begin{equation}\label{L1}
\begin{split}
\norm{\mathcal{F}(v)}_{L^1} & \leq \sqrt{\int_{\R^3} (1+ |\xi|^2)^{-s} \ d\xi }  \ \cdot \sqrt{\int_{\R^3} (1+ |\xi|^2)^{s}\vert \hat{v}(\xi)\vert^2 \ d\xi }  \\
            & \stackrel{\operatorname{def}}{=} C(s)  \norm{v}_{H^{s}}.
\end{split}
\end{equation}
\end{remark}

Now, we show how different choices of $\delta$ and $s$ can provide powerful estimates.

\begin{corollary}\label{corolario1}
Let $u^0, b^0 \in H^{s_{0}}(R^3)$, with $s_{0} > 3/2$, be such that $\Div u^{0} = \Div b^{0} = 0$. Let $u$, $b \in C([0,T^*), H^s(\R^3))$ be the strong solution for the system \eqref{MHD2} defined in its maximal interval $[0, T^*)$. If $T^* < \infty$, then, for each $s\geq 1$,
\begin{equation}\label{consequencia1}
\norm{\big(u,b \big)(t)}_{\dot{H}^s} \norm{\big(u,b\big)(t)}_{L^2}^{s-1} \ge \frac{c(s,\sigma)\lambda^{3s/4}}{(T^*-t)^{s/4}},  \hbox{ for all } t\in[0,T^*).
\end{equation}
\end{corollary}

\begin{proof}
Follows from Theorem \ref{mainresult}$(i)$, by chossing $\delta = 1/2$.
\end{proof}

Next, we obtain a version of Leray's Inequality \cite{Leray1934} for the MHD Equations:

\begin{corollary}\label{corolariooutro}
Let $u^0, b^0 \in H^{s_{0}}(R^3)$, with $s_{0} > 3/2$, be such that $\Div u^{0} = \Div b^{0} = 0$. Let $u$, $b \in C([0,T^*), H^s(\R^3))$ be the strong solution for the system \eqref{MHD2} defined in its maximal interval $[0, T^*)$. If $T^* < \infty$, then,
\begin{equation}
\|(\nabla u, \nabla b)(t)\|_{L^{2}(\R^{3})} \geq \dfrac{c \lambda ^{3/4}}{(T^{\ast}-t)^{1/4}}\hbox{ for all } t\in[0,T^*).
\end{equation}
\end{corollary}

\begin{proof}
It suffices to take $s=1$ in the previous corollary.
\end{proof}

\begin{corollary}\label{corolario2}
Let  $u,b \in C([0,T^*), H^s(\R^3))$ be the strong solution for $(\ref{MHD2})$ defined in its maximal interval $[0, T^*)$. Let $u^0, b^0 \in H^{s_{0}}(\R^3)$, with $s_{0} > 3/2$ and $\Div u^{0} = \Div b^{0} = 0$. If $T^* < \infty$, then, for each $\frac{1}{2}<s<\frac{3}{2}$,
\begin{equation}
\norm{(u,b)(t)}_{\dot{H}^s( \R^3)}\geq\frac{c(s,\sigma)\lambda^{\frac{5}{4}-\frac{s}{2}}}{(T^*-t)^{\frac{s}{2}-\frac{1}{4}}}\hbox{ for all } t\in[0,T^*).
\end{equation}
\end{corollary}

\begin{proof}
This is Theorem \ref{mainresult}$(i)$ with $s = \delta + 1/2$.
\end{proof}

\noindent \textbf{Structure of the paper.} In the next section, we fix the notation, and we state the results needed in the proof of Theorem \ref{mainresult}. The rest of the paper is concerned with the proof of Theorem \ref{mainresult}.

\section{Preliminaries}\label{prelims}

Here, we introduce some notations and state the results that we use in the rest of the paper.

If $u=(u_1,u_2,u_3)$ and $v=(v_1,v_2,v_3)$ are vector fields, the tensor product $u\otimes v$ is
\begin{equation*}
 u\otimes v:=(v_1u,v_2u,v_3u)
\end{equation*}
and its divergence is
\begin{equation*}
 \Div (u\otimes v)=\big(\Div(v_1 u), \Div(v_2 u), \Div(v_3 u)\big).
\end{equation*} It is straightforward to check that, if $u$ is divergence--free, then $$\Div (u\otimes u) = \sum_{i=1}^3 u_i\  \frac{\fr u}{\fr x_i} = (u \cdot \grad) u.$$

The Fourier Transform of $f$ is given by
\begin{equation*}
\mathcal F(f)(\xi) = \hat{f} (\xi) := \int_{\R^3} e^{-i \xi\cdot x} f(x) \ dx.
\end{equation*} We consider the Euclidean norm $\vert z \vert ^{2} = z \cdot z $ in $\mathbb{C}^3$.

For $s\in\mathbb{R}$, the Homogeneous Sobolev Space $\dot{H}^s(\R^3, \R^3)$ is the space of tempered distributions $f$ for which
\begin{equation*}
\Vert f \Vert _{\dot{H}^s} \stackrel{\operatorname{def}}{=} \sqrt{\int_{\R^3} |\xi|^{2s}\vert \hat{f}(\xi)\vert^2 \ d\xi } < +\infty.
\end{equation*} In this way, we note that if $u, b \in \dot{H}^s(\R^3, \R^3)$, we have
\begin{equation*}
\Vert (u,b) \Vert _{\dot{H}^s} = \sqrt{\Vert u \Vert _{\dot{H}^s}^{2}+\Vert b \Vert _{\dot{H}^s}^{2}}.
\end{equation*} The inner product in given by
\begin{equation}\label{innerproduct}
\langle f, g\rangle_{\dot{H}^s} \stackrel{\operatorname{def}}{=} \int_{\R^3} |\xi|^{2s} \hat{f}(\xi) \cdot {\hat{g} (\xi)} \ d\xi,
\end{equation}
where $\hat{f}\cdot \hat{g} = \sum_{i=1}^{3} \hat{f_{i}}\overline{\hat{g_{i}}}$. For Homogeneous Sobolev Spaces, we refer to the book \cite{bahouri_chemin_danchin}. For instance, the following basic interpolation inequality holds:

\begin{lemma}\label{interpolacao}
For $0 < s_0 \le s$, the space $L^2 \cap \dot{H}^{s}$ is a subset of $\dot{H}^{s_0}$, and we have
\[
\norm{f}_{\dot{H}^{s_0}} \le \norm{f}^{1 - s_0/s}_{L^2} \norm{f}_{\dot{H}^{s}}^{s_0/s}.
\]
\end{lemma}
\begin{proof}
This is a particular case of \cite[Proposition 1.32]{bahouri_chemin_danchin}.
\end{proof}

Next, we state a version of what is known as Chemin's Lemma:

\begin{lemma}\label{lema_zingano}
Fix $\delta \in (0,1)$, let $\eta = 1/2 + \delta < 3/2$ and $\eta'= s+1-\delta <3/2$. Then there exists a constant $C(s,\delta) >0$ such that, for every $f,g \in \dot{H}^\eta(\R^3) \cap \dot{H}^{\eta'}(\R^3)$, \[\norm{fg}_{\dot{H}^s} \le C(s,\delta) \norm{f}_{\dot{H}^\eta}\norm{g}_{\dot{H}^{\eta'}}.\]
\end{lemma}

\begin{proof}
Observe $\eta + \eta' = s + 3/2$ and apply Chemin's Lemma \cite[Lemma 3.1]{B2010}.
\end{proof}

Now, we state further interpolation inequalities. For convenience, we present their short proofs.

\begin{lemma}\label{interpolacao_zingano}
Let $\theta$ be a tempered distribution. Then, for $0<\delta<1$ and $s\geq 1/2+\delta$, we have
\[
  \Vert \theta\Vert_{\dot{H}^{s+1-\delta}}\leq\Vert\theta\Vert_{\dot{H}^s}^{\delta}\Vert\nabla\theta\Vert_{\dot{H}^s}^{1-\delta}
 \]
\end{lemma}

\begin{proof}
Since $\widehat{\grad \theta}(\xi) = i \xi\  \hat{\theta}(\xi)$, the proof is a simple application of Hölder's inequality:
\[
\begin{split}
\int_{\R^3} |\xi|^{2(s+1-\delta)} |\hat{\theta}(\xi)|^2 \ d\xi & = \int_{\R^3} |\xi|^{2s\delta} \hat{\theta}(\xi) \cdot |\xi|^{2(s+1)(1-\delta)} \hat{\theta}(\xi) \ d\xi \\
       & \le \bigg(\int_{\R^3} |\xi|^{2s} |\hat{\theta}(\xi)|^2 \ d\xi \bigg)^{\delta}  \bigg(\int_{\R^3}  |\xi|^{2(s+1)} |\hat{\theta}(\xi)|^2 \ d\xi \bigg)^{1-\delta}.
\end{split}
\]\end{proof}

\begin{lemma}\label{ZinganoA8} Let $\theta$ be a tempered distribution and $s > 3/2$. Then,
$$\| \hat{\theta}\|_{L^{1}} \leq C(s)\| \theta\|_{L^{2}}^{1-3/2s}\| \theta\|_{\dot{H}^{s}}^{3/2s}$$
\end{lemma}

\begin{proof}
This follows from \eqref{L1} and a scaling argument.
\end{proof}

\section{Proof of Theorem \ref{mainresult}}\label{main_proof}

In this section, we prove Theorem \ref{mainresult}.

\begin{proof}[Proof of $(i)$]
Since $\Div u = \Div b =0$, by taking the $\dot{H}^s$--scalar product of the first equation in (\ref{MHD2}) with $u$, we obtain
\begin{equation}\label{passo1}
\frac{1}{2}\partial_t\norm{u}^2_{\dot{H}^s}+\nu\norm{\nabla u}^2_{\dot{H}^s} - Re[\langle u\otimes u, \grad u\rangle_{\dot{H}^s}] + Re[\sigma \langle  b\otimes b, \grad u\rangle_{\dot{H}^s}]=0,
\end{equation} where $Re[z]$ denotes the real part of the complex number $z$. Then,
\begin{eqnarray}\label{inicial}
\frac{1}{2}\partial_t\norm{u}^2_{\dot{H}^s}+\nu\norm{\nabla u}^2_{\dot{H}^s}&=& Re[ \langle u\otimes u  - \sigma  ( b\otimes b), \grad u\rangle_{\dot{H}^s}] \nonumber\\
&\leq& \norm{u\otimes u-\sigma (b\otimes b)}_{\dot{H}^s}\norm{\nabla u}_{\dot{H}^s}\\
&\leq& \norm{u\otimes u}_{\dot{H}^s}\norm{\nabla u}_{\dot{H}^s}+\sigma\norm{b\otimes b}_{\dot{H}^s}\norm{\nabla u}_{\dot{H}^s}. \nonumber
\end{eqnarray}

\noindent Now, we apply Lemma \ref{lema_zingano} and Lemma \ref{interpolacao_zingano} to the right hand side of \eqref{inicial} to obtain
\begin{equation}\label{inicialu}
\begin{split}
\frac{1}{2}\partial_t\norm{u}^2_{\dot{H}^s}  & +\nu\norm{\nabla u}^2_{\dot{H}^s} \leq  c(s,\delta) \Big[\norm{u}_{\dot{H}^{1/2 + \delta}}\norm{u}_{\dot{H}^{s+ 1 - \delta}}  +  \sigma\norm{b}_{\dot{H}^{1/2 + \delta}}\norm{b}_{\dot{H}^{s+ 1 - \delta}}\Big]\norm{\nabla u}_{\dot{H}^s} \\
                      & \le  c(s,\delta,\sigma) \Big[\norm{u}_{\dot{H}^{1/2 + \delta}}\norm{u}_{\dot{H}^{s}}^\delta \norm{\grad u}_{\dot{H}^{s}}^{2-\delta} + \norm{b}_{\dot{H}^{1/2 + \delta}}\norm{b}_{\dot{H}^{s}}^\delta \norm{\grad b}_{\dot{H}^{s}}^{1-\delta}\norm{\nabla u}_{\dot{H}^s}\Big].
\end{split}
\end{equation}

\noindent Thus, Young's Inequality implies
\begin{equation}\label{estimativa_u}
\begin{split}
\partial_t\norm{u}^2_{\dot{H}^s} + \nu \norm{\nabla u}^2_{\dot{H}^s} &  \leq c(s,\delta,\sigma) \Big[ \nu^{-\frac{2-\delta}{\delta}} \norm{u}^{2/\delta}_{\dot{H}^{1/2 + \delta}}\norm{u}^{2}_{\dot{H}^s}  + \nu^{-1} \norm{b}^2_{\dot{H}^{1/2 + \delta}}\norm{b}_{\dot{H}^{s}}^{2\delta} \norm{\grad b}_{\dot{H}^{s}}^{2-2\delta} \Big].
\end{split}
\end{equation}

\noindent From the second equation of the system \eqref{MHD2}, we obtain an analogous estimate:
\begin{equation}\label{estimativa_b}
\begin{split}
\partial_t\norm{b}^2_{\dot{H}^s}+\nu_m\norm{\nabla b}^2_{\dot{H}^s}  & \leq  c(s, \delta) \bigg[ \nu_m^{-1} \norm{b}^2_{\dot{H}^{1/2 + \delta}}\norm{u}_{\dot{H}^{s}}^{2\delta} \norm{\grad u}_{\dot{H}^{s}}^{2-2\delta} + \nu_m^{-\frac{2-\delta}{\delta}} \norm{u}^{2/\delta}_{\dot{H}^{1/2 + \delta}}\norm{b}_{\dot{H}^{s}}^2 \bigg]
\end{split}
\end{equation}

\noindent Now, we set $\lambda = \min\{ \nu, \nu_m\}$ and add \eqref{estimativa_u} and \eqref{estimativa_b}. We obtain
\[
\begin{split}
\partial_t\norm{(u,b)}^2_{\dot{H}^s} + \lambda \norm{(\grad u, \nabla b)}^2_{\dot{H}^s}
            & \le c(s, \delta,\sigma) \bigg[ \lambda^{-\frac{2-\delta}{\delta}} \norm{u}^{2/\delta}_{\dot{H}^{1/2 + \delta}} \big( \norm{u}^{2}_{\dot{H}^s} + \norm{b}^{2}_{\dot{H}^s} \big)  \\
            & \quad + \lambda^{-1} \norm{b}^2_{\dot{H}^{1/2 + \delta}}\norm{b}_{\dot{H}^{s}}^{2\delta} \norm{\grad b}_{\dot{H}^{s}}^{2-2\delta} + \lambda^{-1} \norm{b}^2_{\dot{H}^{1/2 + \delta}}\norm{u}_{\dot{H}^{s}}^{2\delta} \norm{\grad u}_{\dot{H}^{s}}^{2-2\delta} \bigg].
\end{split}
\]

By Young's inequality again, it follows that
\[
\begin{split}
\partial_t\norm{(u,b)}^2_{\dot{H}^s} + \frac{\lambda}{2} \norm{(\grad u, \nabla b)}^2_{\dot{H}^s} & \le  c(s, \delta,\sigma) \lambda^{-\frac{2-\delta}{\delta}} \bigg[ \norm{u}^{2/\delta}_{\dot{H}^{1/2 + \delta}} \norm{(u,b)}^{2}_{\dot{H}^s}  + \norm{b}^{2/\delta}_{\dot{H}^{1/2 + \delta}} \norm{(u,b)}^{2}_{\dot{H}^s} \bigg]  \\
            & = c(s, \delta,\sigma) \lambda^{-\frac{2-\delta}{\delta}} \Big[ \norm{u}^{2/\delta}_{\dot{H}^{1/2 + \delta}} + \norm{b}^{2/\delta}_{\dot{H}^{1/2 + \delta}} \Big] \norm{(u,b)}^{2}_{\dot{H}^s}.
\end{split}
\] In particular,
\[
\partial_t\norm{(u,b)}^2_{\dot{H}^s} \le c(s, \delta,\sigma) \lambda^{-\frac{2-\delta}{\delta}} \Big[ \norm{u}^{2/\delta}_{\dot{H}^{1/2 + \delta}} + \norm{b}^{2/\delta}_{\dot{H}^{1/2 + \delta}} \Big] \norm{(u,b)}^{2}_{\dot{H}^s}.
\] Thus, by Gronwall's Inequality with $0\leq a \leq t<T^*$,
\begin{equation}\label{nova1}
\norm{\big(u,b\big)(t)}^2_{\dot{H}^s}  \leq \norm{\big(u,b\big)(a)}^2_{\dot{H}^s} \exp \bigg( c(s, \delta,\sigma) \lambda^{-\frac{2-\delta}{\delta}}  \int_a^t \Big[ \norm{u(\tau)}^{2/\delta}_{\dot{H}^{1/2 + \delta}} + \norm{b(\tau)}^{2/\delta}_{\dot{H}^{1/2 + \delta}} \Big] \ d\tau \bigg).
\end{equation}

\noindent Set $\theta : = (1 + 2\delta)/2s.$ By Lemma \ref{interpolacao}, whenever $s \ge 1/2 + \delta$ then
\begin{equation}\label{lemma2.3}
\norm{f}_{\dot{H}^{1/2 + \delta}} \le \norm{f}^{1 - \theta}_{L^2} \norm{f}_{\dot{H}^{s}}^{\theta}.
\end{equation}

\noindent Estimate \eqref{lemma2.3}, applied to $u$ and $b$, yields
\[
\begin{split}
\norm{u}^{2/\delta}_{\dot{H}^{1/2 + \delta}} + \norm{b}^{2/\delta}_{\dot{H}^{1/2 + \delta}} & \le 2^{1/\delta} \big( \norm{u}^{2}_{\dot{H}^{1/2 + \delta}} + \norm{b}^{2}_{\dot{H}^{1/2 + \delta}} \big)^{1/\delta} \\
           & \le 2^{1/\delta} \Big[ \norm{u}^{2 - 2\theta}_{L^2} \norm{u}_{\dot{H}^{s}}^{2\theta} + \norm{b}^{2 - 2\theta}_{L^2} \norm{b}_{\dot{H}^{s}}^{2\theta} \Big]^{1/\delta} \\
           & \le 2^{1/\delta} \Big[ \big( \norm{u}^{2 - 2\theta}_{L^2} + \norm{b}^{2 - 2\theta}_{L^2} \big) \norm{u}_{\dot{H}^{s}}^{2\theta} + \big( \norm{u}^{2 - 2\theta}_{L^2} + \norm{b}^{2 - 2\theta}_{L^2} \big) \norm{b}_{\dot{H}^{s}}^{2\theta} \Big]^{1/\delta}  \\
           & = 2^{1/\delta} \Big[ \norm{u}^{2 - 2\theta}_{L^2} + \norm{b}^{2 - 2\theta}_{L^2} \Big]^{1/\delta} \Big[ \norm{u}_{\dot{H}^{s}}^{2\theta} + \norm{b}_{\dot{H}^{s}}^{2\theta} \Big]^{1/\delta} \\
           & \le 2^{1/\delta} \Big[ \norm{u}^{2}_{L^2} + \norm{b}^{2}_{L^2} \Big]^{(1-\theta)/\delta} \Big[ \norm{u}_{\dot{H}^{s}}^{2} + \norm{b}_{\dot{H}^{s}}^{2} \Big]^{\theta/\delta}
\end{split}
\]

\noindent This and (\ref{nova1}) imply
\[
\begin{split}
\norm{u(t)}^{2/\delta}_{\dot{H}^{1/2 + \delta}} + \norm{b(t)}^{2/\delta}_{\dot{H}^{1/2 + \delta}} & \le 2^{1/\delta} \norm{\big(u,b\big)(t)}^{(2-2\theta)/\delta}_{L^2} \norm{\big(u,b\big)(a)}^{2\theta/\delta}_{\dot{H}^s} \times \\
        & \quad \times \exp \bigg( c(s, \delta,\sigma) \lambda^{-\frac{2-\delta}{\delta}}  \int_a^t \Big[ \norm{u(\tau)}^{2/\delta}_{\dot{H}^{1/2 + \delta}} + \norm{b(\tau)}^{2/\delta}_{\dot{H}^{1/2 + \delta}} \Big] \ d\tau \bigg)
\end{split}
\]
Then, by using the decay of the $L^2$--norm, we obtain
\[
\begin{split}
f(t)\exp \bigg(- c(s, \delta,\sigma) \lambda^{-\frac{2-\delta}{\delta}}  \int_a^t f(\tau) \ d\tau \bigg) & \le 2^{1/\delta} \norm{\big(u,b\big)(a)}_{L^2}^{(2-2\theta)/\delta} \norm{\big(u,b\big)(a)}^{2\theta/\delta}_{\dot{H}^s}
\end{split}
\] where \[f(t) := \norm{u(t)}^{2/\delta}_{\dot{H}^{1/2 + \delta}} + \norm{b(t)}^{2/\delta}_{\dot{H}^{1/2 + \delta}}
\] Now, we integrate on $[a, T]$, with $T<T^*$:
\[
\begin{split}
\int_a^T f(t)  \exp \bigg(- c(s, \delta,\sigma) \lambda^{-\frac{2-\delta}{\delta}} & \int_a^t f(\tau) \ d\tau \bigg) \ dt \le 2^{1/\delta} \norm{\big(u,b\big)(a)}_{L^2}^{(2-2\theta)/\delta} \norm{\big(u,b\big)(a)}^{2\theta/\delta}_{\dot{H}^s} (T-a);
\end{split}
\] thus,
\[
\begin{split}
\frac{-\lambda^{\frac{2-\delta}{\delta}}}{c(s, \delta,\sigma)} \exp \bigg(- c(s, \delta,\sigma) \lambda^{-\frac{2-\delta}{\delta}} & \int_a^t f(\tau) \ d\tau \bigg) \bigg|_{t=a}^{t=T} \le 2^{1/\delta} \norm{\big(u,b\big)(a)}_{L^2}^{(2-2\theta)/\delta} \norm{\big(u,b\big)(a)}^{2\theta/\delta}_{\dot{H}^s} (T-a),
\end{split}
\] or,
\[
\begin{split}
1-\exp \bigg(- c(s, \delta,\sigma) \lambda^{-\frac{2-\delta}{\delta}} & \int_a^T f(\tau) \ d\tau \bigg) \le c(s, \delta,\sigma) \lambda^{-\frac{2-\delta}{\delta}} \norm{\big(u,b\big)(a)}_{L^2}^{(2-2\theta)/\delta} \norm{\big(u,b\big)(a)}^{2\theta/\delta}_{\dot{H}^s} (T-a).
\end{split}
\] Let $T \to T^*$ and use \eqref{nova1}: 
\[
\begin{split}
1  \le c(s, \delta,\sigma) \lambda^{-\frac{2-\delta}{\delta}} \norm{\big(u,b\big)(a)}_{L^2}^{(2-2\theta)/\delta} \norm{\big(u,b\big)(a)}^{2\theta/\delta}_{\dot{H}^s} (T^*-a).
\end{split}
\]

Hence,
\begin{equation}\label{final_in}
\norm{\big(u,b\big)(a)}_{\dot{H}^s} \norm{\big(u,b\big)(a)}_{L^2}^{p(s,\delta)} \ge \frac{c(s, \delta,\sigma)\lambda^{q(s,\delta)}}{(T^*-a)^{r(s,\delta)}}.
\end{equation}
\end{proof}

\begin{proof}[Proof of $(ii)$]
Set $q= p + \frac{1}{2}\sigma |b|^2$. We apply the Fourier Transform to both sides of the first equation in \eqref{MHD2} to obtain
\begin{equation}
\fr_t \big( \mathcal{F}(u) \big)  + \mathcal{F}\big((u \cdot \grad) u\big) - \sigma \mathcal{F}\big((b \cdot \grad) b\big) = \nu \mathcal{F}(\Lap u) - \mathcal{F}(\grad q).
\end{equation} Since $\Div u = 0$, we have $\mathcal{F}(\grad q)\cdot\mathcal{F}(u)=0$. Thus,
\begin{equation}\label{fourier}
\frac{1}{2}\fr_t \Big( |\mathcal{F}(u)|^2  \Big)  + \nu|\xi|^2|\mathcal{F}(u)|^2 + Re[\mathcal{F}\big((u \cdot \grad) u\big) \cdot \mathcal{F}(u)] - Re[\sigma  \mathcal{F}\big((b \cdot \grad) b\big)\cdot \mathcal{F}(u) ] =0.
 \end{equation} Observe that, for any fixed $\eps>0,$
\begin{equation*}
\frac{1}{2}\ \fr_t \Big( |\mathcal{F}(u)|^2  \Big) = \sqrt{ |\mathcal{F}(u)|^2+\varepsilon} \ \ \fr_t \Big( \sqrt{|\mathcal{F}(u)|^2+\varepsilon}\ \Big).
\end{equation*} This and \eqref{fourier} imply
\begin{equation}
\fr_t \Big(\sqrt{|\mathcal{F}(u)|^2 + \epsilon} \ \Big) + \frac{\nu|\xi|^2|\mathcal{F}(u)|^2 }{\sqrt{|\mathcal{F}(u)|^2 + \epsilon}} +  \frac{Re[\mathcal{F}\big((u\cdot\grad) u\big) \cdot \mathcal{F}(u)]}{\sqrt{|\mathcal{F}(u)|^2 + \epsilon}} - \frac{Re[\sigma \mathcal{F}\big((b\cdot\grad)b\big) \cdot\mathcal{F}(u)]}{\sqrt{|\mathcal{F}(u)|^2 + \epsilon}} = 0.
\end{equation} Hence, by setting $\kappa= \max \{1,\sigma\}$, we have
\begin{equation}
\fr_t \Big(\sqrt{|\mathcal{F}(u)|^2 + \epsilon} \ \Big) + \frac{\nu|\xi|^2| \mathcal{F}(u)|^2}{\sqrt{|\mathcal{F}(u)|^2 + \epsilon}} \leq \kappa \Big[ \big|\mathcal{F}\big((u\cdot\grad)u\big) \big| + \big|\mathcal{F}\big((b\cdot\grad)b\big)\big| \Big].
\end{equation} Let $\varepsilon \to 0$:
\begin{equation}
 \fr_t\big(|\mathcal{F}(u)|\big) + \nu |\xi|^2 |\mathcal{F}(u)| \leq \kappa \Big[ \big|\mathcal{F}\big((u\cdot\grad)u\big) \big| + \big|\mathcal{F}\big((b\cdot\grad)b\big)\big| \Big].
 \end{equation} Now, integrate on $\R^3$ to obtain
\begin{equation*}
\begin{split}
 \fr_t\norm{\mathcal{F}(u)(t)}_{L^1}  + & \nu\norm{\mathcal{F}(\Delta u)(t)}_{L^1} \le \kappa \left[\int_{\mathbb{R}^3} \big|\mathcal{F}\big((u\cdot\grad)u\big)(t,\xi)\big| \ d\xi + \int_{\mathbb{R}^3} \big|\mathcal{F}\big((b\cdot\grad)b\big)(t,\xi) \big| \ d\xi\right] \\
    &  \le \frac{\kappa3\sqrt{3}}{(2\pi)^3} \Big[ \norm{\mathcal{F}(u)(t)}_{L^1}^{3/2} \norm{\mathcal{F}(\Delta u)(t)}_{L^1}^{1/2} + \norm{\mathcal{F}(b)(t)}_{L^1}^{3/2} \norm{\mathcal{F}(\Delta b)(t)}_{L^1}^{1/2} \Big].
\end{split}
\end{equation*}

\noindent Now, apply Young's Inequality twice to obtain
\begin{equation*}\label{28}
\begin{split}
\fr_t\norm{\mathcal{F}(u)(t)}_{L^1} + \frac{\nu}{2} \norm{\mathcal{F}(\Delta u)(t)}_{L^1} \le & \ \frac{1}{2} \left(\frac{\kappa3\sqrt{3}}{(2\pi)^3}\right)^2 \big[\nu^{-1}\norm{\mathcal{F}(u)(t)}_{L^1}^3+\nu_m^{-1} \norm{\mathcal{F}(b)(t)}_{L^1}^3\big]
 \\ &+ \frac{\nu_m}{2}\norm{\mathcal{F}(\Delta b)(t)}_{L^1}.
\end{split}
\end{equation*} Let $\lambda=\min\{\nu,\nu_m\}$. Then,
\begin{equation}\label{inequality_u}
\begin{split}
\fr_t \big( \norm{\mathcal{F}(u)(t)}_{L^1} \big) + \frac{\nu}{2} \norm{\mathcal{F}(\Delta u)(t)}_{L^1}
\leq & \frac{27\lambda^{-1}\kappa^2}{2(2\pi)^6} \big[\norm{\mathcal{F}(u)(t)}_{L^1}^3 + \norm{\mathcal{F}(b)(t)}_{L^1}^3 \big]+ \frac{\nu_m}{2}\norm{\mathcal{F}(\Delta b)(t)}_{L^1}.
\end{split}
\end{equation}

Analogously, from the second equation in the MHD system \eqref{MHD2}, we obtain

\begin{equation}\label{inequality_b}
\fr_t \big(\norm{\mathcal{F}(b)(t)}_{L^1} \big) + \frac{\nu_m}{2} \norm{\mathcal{F}(\Delta b)(t)}_{L^1}
\le \frac{27\lambda^{-1}}{2(2\pi)^6} \big[\norm{\mathcal{F}(u)(t)}_{L^1}^3+\norm{\mathcal{F}(b)(t)}_{L^1}^3 \big] + \frac{\nu}{2}\norm{\mathcal{F}(\Delta u)(t)}_{L^1}.
\end{equation} By \eqref{inequality_u} and \eqref{inequality_b}, it follows that
\begin{eqnarray}
\fr_t\Big(\norm{\mathcal{F}(u)(t)}_{L^1} + \norm{\mathcal{F}(b)(t)}_{L^1}\Big) &\leq& \frac{27\kappa^2\lambda^{-1}}{(2\pi)^6} \Big[\norm{\mathcal{F}(u)(t)}_{L^1}^3+\norm{\mathcal{F}(b)(t)}_{L^1}^3\Big],
\end{eqnarray} whence
\begin{eqnarray}
\fr_t\norm{\mathcal{F}(u,b)(t)}_{L^1}&\leq& \frac{27\kappa^2\lambda^{-1}}{(2\pi)^6}\norm{\mathcal{F}(u,b)(t)}_{L^1}^3.
\end{eqnarray}

Set $K= 27\kappa^2 /(2\pi)^6$. Then, for $t_{0} \in [0, T^*)$, consider a solution $v(t)$ to the initial value problem
\begin{equation}\label{v}
\left\{
\begin{split}
v'(t)&=K \lambda ^{-1}v^3(t) \\
v(t_{0})&= \norm{\mathcal{F}(u,b)(t_{0})}_{L^1}.
\end{split}
\right.
\end{equation} By integration, we have, for all $t \in [t_{0}, T^*)$,
\begin{equation}
v(t)^{-2}= {\norm{\mathcal{F}(u,b)(t_0)}_{L^1}^{-2}} -2K \lambda^{-1} (t-t_{0}).
\end{equation} It follows that $v(t)$ explodes if
$$T^{*}_{v}= t_0+\frac{\lambda}{2K{\norm{\mathcal{F}(u,b)(t_0)}_{L^1}^{2}}}.$$
Since ${\norm{\mathcal{F}(u,b)(t)}_{L^1}}\leq v(t)$, we conclude $T^{*}_{v}\leq T^*.$ Therefore,
\begin{equation*}
\norm{\mathcal{F}(u,b)(t_0)}_{L^1} \geq \frac{ (2\pi)^{3}\lambda^{1/2} }{3\sqrt{6} \kappa} {(T^{*}-t_0)^{-1/2}}.
\end{equation*}
\end{proof}

\begin{proof}[Proof of $(iii)$]
By Lemma \ref{ZinganoA8}, there holds, for $s>3/2$,
\begin{equation}
\norm{\mathcal{F}(u)}_{L^1}\leq c_1(s)\norm{u}_{L^2}^{1-\frac{3}{2s}}\norm{u}_{\dot{H}^s}^{\frac{3}{2s}}
\end{equation} and
\begin{equation}
\norm{\mathcal{F}(b)}_{L^1}\leq c_2(s)\norm{b}_{L^2}^{1-\frac{3}{2s}}\norm{b}_{\dot{H}^s}^{\frac{3}{2s}}.
\end{equation} Add these to obtain
\begin{eqnarray}
\norm{\mathcal{F}(u)}_{L^1}+\norm{\mathcal{F}(b)}_{L^1}&\leq& c(s)[\norm{u}_{L^2}^{1-\frac{3}{2s}}\norm{u}_{\dot{H}^s}^{\frac{3}{2s}}+\norm{b}_{L^2}^{1-\frac{3}{2s}}\norm{b}_{\dot{H}^s}^{\frac{3}{2s}}]\nonumber\\
&\leq& c(s)[(\norm{u}_{L^2}^{1-\frac{3}{2s}}+\norm{b}_{L^2}^{1-\frac{3}{2s}})\norm{u}_{\dot{H}^s}^{\frac{3}{2s}}+(\norm{u}_{L^2}^{1-\frac{3}{2s}}+\norm{b}_{L^2}^{1-\frac{3}{2s}})\norm{b}_{\dot{H}^s}^{\frac{3}{2s}}]\nonumber\\
&\leq& c(s)[(\norm{u}_{L^2}^{1-\frac{3}{2s}}+\norm{b}_{L^2}^{1-\frac{3}{2s}})(\norm{u}_{\dot{H}^s}^{\frac{3}{2s}}+\norm{b}_{\dot{H}^s}^{\frac{3}{2s}})],
\end{eqnarray} where $c(s)=c_1(s)+c_2(s)$.

But $\norm{u}_{L^2}^2+\norm{b}_{L^2}^2=\norm{(u,b)}_{L^2}^2$. Thus, $\norm{(u,b)}_{L^2}^2\geq \norm{u}_{L^2}^2$, which implies
$$
\norm{u}_{L^2}^{1-\frac{3}{2s}}\leq\norm{(u,b)}_{L^2}^{1-\frac{3}{2s}}
$$
for $s>3/2$ $\left(1-\frac{3}{2s}>0\right)$.

Analogously,
$$
\norm{b}_{L^2}^{1-\frac{3}{2s}}\leq\norm{(u,b)}_{L^2}^{1-\frac{3}{2s}}
$$
for $s>3/2$.

Similarly, since $\norm{u}_{\dot{H}^s}^2+\norm{b}_{\dot{H}^s}^2=\norm{(u,b)}_{\dot{H}^s}^2$, we have
$$
\norm{u}_{\dot{H}^s}^{\frac{3}{2s}}\leq\norm{(u,b)}_{\dot{H}^s}^{\frac{3}{2s}}\,\,\,\mbox{e}\,\,\,\norm{b}_{\dot{H}^s}^{\frac{3}{2s}}\leq\norm{(u,b)}_{\dot{H}^s}^{\frac{3}{2s}}.
$$

Then,
\begin{eqnarray}
\norm{\mathcal{F}(u,b)}_{L^1}\leq\norm{\mathcal{F}(u)}_{L^1}+\norm{\mathcal{F}(b)}_{L^1}&\leq& c(s)\norm{(u,b)}_{L^2}^{1-\frac{3}{2s}}\norm{(u,b)}_{\dot{H}^s}^{\frac{3}{2s}}.
\end{eqnarray}

By Theorem \ref{mainresult}(ii),
\begin{equation}
\left[\frac{\kappa}{3\sqrt{6}}(2\pi)^{3}\lambda^{1/2}{(T^*-t)^{-1/2}}\right]^{2s/3}\leq\left[c(s)\norm{(u,b)(t)}_{L^2}^{1-\frac{3}{2s}}\norm{(u,b)(t)}_{\dot{H}^s}^{\frac{3}{2s}}\right]^{2s/3},
\end{equation}
where $\kappa = \max\{1,\sigma\}^{-1}.$ Thus,
\begin{equation}
\Big[\frac{\kappa}{3\sqrt{6}}(2\pi)^{3}\lambda^{1/2}\Big]^{2s/3}(T^*-t)^{-s/3}\leq c(s)^{2s/3}\norm{(u,b)(t)}_{L^2}^{\frac{2s}{3}-1}\norm{(u,b)(t)}_{\dot{H}^s}.
\end{equation}

Therefore,
\begin{equation*}
\frac{c(s,\sigma) \lambda^{s/3}}{(T^*-t)^{s/3}} \leq \norm{(u,b)(t)}_{L^2}^{\frac{2s}{3}-1} \norm{(u,b)(t)}_{\dot{H}^s},
\end{equation*} as we wanted.
\end{proof}

\noindent \textbf{Acknowledgements.} We thank Prof. Paulo Zíngano for comments and fruitful discussions. D. Marcon was partially supported by CNPq -- Brazil through the postdoctoral scholarship PDJ/501839/2013-5. W. Melo was partially supported by CAPES (Ciência sem Fronteiras) grant 2778-13-0. W. Melo would like to thank the Federal University of Sergipe for supporting his long term visit to the University of New Mexico during the academic year 2013-2014.

\signdm

\signwm

\signls

\signjz

\end{document}